\documentclass[11pt]{amsart}
\usepackage{amsmath,amssymb, graphicx, amscd,latexsym,comment,color}
\makeatletter
\newtheorem{Theorem}{Theorem}

\newtheorem{Corollary}[Theorem]{Corollary}

\newcommand{\eps}{\varepsilon}

\newcommand\al{\alpha}

\newcommand\si{\sigma}
\newcommand\be{\beta}

\newcommand\ga{\gamma}

\newcommand\de{\delta}
\newcommand\De{\Delta}

\begin{document}
\title[ Remark on the  roots of  generalized Lens equations 
]
{Remark on the  roots of  generalized Lens equations }

\author
[M. Oka ]
{Mutsuo Oka }
\address{\vtop{
\hbox{Department of Mathematics}
\hbox{Tokyo  University of Science}
\hbox{1-3 Kagurazaka, Shinjuku-ku}
\hbox{Tokyo 162-8601}
\hbox{\rm{E-mail}: {\rm oka@rs.kagu.tus.ac.jp}}
}}

\begin{abstract}
We consider roots of   a generalized   Lens polynomial
$L(z,\bar z)=\bar z^m q(z)-{p(z)} $ and also harmonically splitting Lens type polynomial
$L^{hs}(z,\bar z)=r(\bar z) q(z) - p(z)$ with $\deg\, q(z)=n,\,\deg\,p(z)\le n$.
We have shown that there exists a harmonically splitting polynomial 
$r(\bar z)q(z)-p(z)$ which  take $5n+m-6$ roots, using a bifurcation family of polynomial. In this note, we show that this number can be taken by  a generalized Lens polynomial ${\bar z}^mq(z)-p(z)$ after a slight modification of the bifurcation family of a Rhie polynomial.
\end{abstract}
\maketitle

\maketitle

\section{Introduction}
Consider  a mixed polynomial of one  variable $f(z,\bar z)=\sum_{\nu,\mu}a_{\nu,\mu}z^\nu{\bar z}^\mu$.
 We consider the number  of roots of $f=0$.
 Assume that $z=\alpha$ is an isolated zero of $f=0$. 
Put $f(z,\bar z)=g(x,y)+ih(x,y)$ with $z=x+iy$ where $g=\Re( f)$ and $h=\Im (f)$. 
 We call $\al$ a positive simple root (respectively a negative simple root), if the Jacobian $J(g,h)$ is positive (resp. negative)  at $z=\al$.
\subsection{Number of roots with sign}

Let $f(z,\bar z)$ be a given mixed polynomial of one variable,
 we consider the filtration by the degree:
\[\begin{split}
 f(z,\bar z)&=f_{ d}(z,\bar z)+f_{ d-1}(z,\bar z)+\cdots+
f_{0}(z,\bar z).
\end{split}
\]
Here
  $
  f_\ell(z,\bar z):=\sum_{\nu+\mu=\ell}c_{\nu,\mu}z^\nu\bar z^\mu.
 $
We consider the case 
$f_d(z,\bar z)=z^n{\bar z}^m$ with $n+m=d$.
The total number of roots of $f(z,\bar z)=0$ with sign is denoted by $\beta(f)$.
{ \em Under the above assumption,
$\be(f)=n-m$} by Theorem 20, \cite{MixIntersection}.


\subsection{Number of roots without  the sign}
We assume that roots of $f(z,\bar z)=0$ are all simple.
The number of roots without considering the sign is denoted by $\rho(f)$.
Note that  $\rho(f)$ is not described by the highest degree part $f_d$, which was the case for
 $\be(f)$.
Consider  a mixed polynomial $f(z,\bar z)=\sum_{\nu,\mu}a_{\nu,\mu}z^\nu{\bar z}^\mu$.  We use the definitions
\[\begin {split}
&\deg_z\,f:=\max\{\nu\,|\, a_{\nu,\mu}\ne 0\}\\
&\deg_{\bar z}\,f:=\max\{\mu\,|\,a_{\nu,\mu}\ne 0\}\\
&\deg\,f:=\max\{\mu+\nu\,|\,  a_{\nu,\mu}\ne 0\}
\end{split}
\]
 $\deg_z\,f,\,   \deg_{\bar z}\,f,\, \deg\,f$ are called  {\em the holomorphic degree} , {\em  the anti-holomorphic degree}
and {\em the mixed degree} of $f$
 respectively. 
We  consider the following subclasses  of mixed polynomials:
\begin{eqnarray*}
L(n+m;n,m)&:=&\{f(z,\bar z)=\bar z^m q(z)-p(z)\,|\, \deg_z q=n,\,\deg_z p\le n\},\\
L^{hs}(n+m;n,m)&:=&
\{f(z,\bar z)=r(\bar z) q(z)-p(z)\,| \deg_{\bar z} r(\bar z)=m,\,\\
&&\qquad\qquad\deg_z q=n,\,\,\deg_z p\le n\},\\
M(n+m;n,m)&:=&\{f(z,\bar z)\,|\, \deg\,f=n+m,\,\deg_z\,f=n,\,\deg_{\bar z}\,f=m\}.\\
\end{eqnarray*}
where $p(z),q(z)\in \mathbb C[z],r(\bar z)\in \mathbb C[\bar z]$. 
Here $z$ is an affine  coordinate of $\mathbb C$ but  we do not fix $z$.
So, a mixed polynomial $f(u,\bar u)$ is called a generalized Lens polynomial or a harmonically splitting Lens type polynomial if $f$ takes the above form under some affine  coordinate $u=z+c$.
 We have  canonical inclusions:
\[
L(n+m;n,m)\,\,\subset L^{hs}(n+m;n,m)\,\,\subset M(n+m;n,m).
\]
The class $L(n+m;n,m), L^{hs}(n+m;n,m)$ corresponds to the numerators of  harmonic functions 
\[
\bar z^m-\frac{p(z)}{q(z)},\,\, r(\bar z)-\frac{p(z)}{q(z)}.
\]
In particular, $L(n+1;n,1)$ corresponds to the lens equation.
We call $\bar z^m q(z)-p(z)$ 
{\em a generalized lens polynomial}
and and  $ r(\bar z) q(z)-p(z)$  {\em a harmonically splitting lens type polynomial} respectively.

\subsection{Lens equation 
}
The following equation is known as the lens equation.
\begin{eqnarray}\label{lens-rational1}
L(z,\bar z)=\bar z-\sum_{i=1}^n  \dfrac{\si_i}{z-\al_i}=0,\quad \si_i,\al_i\in \mathbb C^*.
\end{eqnarray}
We identify the left side rational function  with the mixed polynomial given by its numerator
\[
\tilde L(z,\bar z):=L(z,\bar z)\prod_{i=1}^n(z-\al_i)\in M(n+1;n,1).
\]

\begin{Theorem} \label{bound}
 \rm (Khavinson-Neumann 
  \cite{Khavinson-Neumann})
 The number of roots  of $L$ or $\tilde L$ is bounded by $5n-5$ for $n\ge 2$.
 \end{Theorem}
  Rhie gave an explicit polynomial which takes this bound $5n-5$ 
 in \cite{Rhie}. Thus this bound is optimal.
On the other hand, $\rho(L)\equiv n-1$ mod $2$ by Theorem 20, \cite{MixIntersection}.
\begin{Theorem}
{\rm (P. Bleher, Y. Homma, L. Ji and P. Roeder  
\cite{Bleher} )} The set of possible values of $\rho(f)$ for $f\in L(n+1,n,1)$ is equal to
 $\{n-1,n+1,\cdots, 5n-5\}$.
\end{Theorem} 


\subsection{Bifurcation family}
In \cite{OkaLens}, we have constructed a generalized Lens type polynomial
which
 take $5(n-m)$-roots if $n>3m$ and we have asked {\em if this is an optimal upper bound or not}. On the other  hand,  
 for the space of harmonically splitting Lens type polynomials  $L^{hs}(n+m;n,m)$, we studied a bifurcation family 
 $\psi_t(z,\bar z):=t{\bar z}^m+\ell_n(z,\bar z)\in L^{hs}(n+m;n,m)$ starting from a given Lens polynomial  $\ell_n(z,\bar z)$ with 
 $\rho(\ell_n)=k$. Let $\al_1,\dots, \al_k$ be the roots of $\ell_n$. We have proved 
 \begin{Theorem}{\rm(\cite{OkaLens})}  $\psi_t=0$ has 
 exactly $k+m-1$ roots for small $t$. 
 Furthermore
 $k$ roots of them are near each $\al_j$ with the same sign and $m-1$ roots
 are newly born roots bifurcated from $z=\infty$. These new roots are negative roots.
 \end{Theorem}
 \section{Main result}
 \subsection{Modification of the bifurcation family and the main result}
 In this note, we answer the above question negatively. In fact, we modify the above bifurcation family
 to prove the same assertion for generalized Lens polynomials.
 We start from an arbitrary Lens type polynomial with only simple roots:
 \[
 \ell_n(z,\bar z):=\bar z q(z)-p(z),\quad \deg_z q=n,\, \deg_z p\le n,\,n\ge 2.
\] Put $k=\rho(\ell_n)$ and let $\al_1,\dots, \al_k$ be the roots of $\ell_n$.
Note that $n-1\le k\le 5n-5$ and $k\equiv n-1$ mod $2$.
Put $\ga$ be the coefficient of  $z^n$ in $q(z)$. $\ga$ is non-zero.
 Consider its small purturbation of $\ell_n(z)$ in the space of generalized Lens polynomials $L(n+m;n,m)$:
\begin{eqnarray}\label{eq-phi-t}
\phi_t(z,\bar z):=\frac{((t \bar z+\ga)^m -\ga^m)}{\ga^{m-1}mt} q(z)-p(z),\,t\in \mathbb C.
\end{eqnarray}
Note that $\phi_0(z,\bar z)=\ell_n(z,\bar z)$ and 
$\phi_t,\,t\ne 0$, corresponds to the generalized Lens equation
\[(t \bar z+\ga)^m q(z)=\ga^{m-1}mtp(z)+\ga^{m} q(z).
\]
In fact, by the change of coordinate $u=\bar t z+\bar \ga$,  $\phi_t$ takes the expected form.
\begin{Theorem}
For sufficiently small $t\in \mathbb C,\,|t|\ll 1$, $\rho(\phi_t)=k+m-1$.
Furthermore 
\begin{enumerate}
\item
$k$ roots  $\al_{j}(t), j=1,\dots, k$ are small deformation of $\al_j,\,j=1,\dots, k$ and the sign of 
$\al_{j}(t)$ is the same as that of $\al_j$.
\item $m-1$ new roots $\be_a(t), a=1,\dots, m-1$ are born at infinity i.e., $\be_a(0)=\infty$
and they are negative roots.
\end{enumerate}
Taking $\ell_n(z,\bar z)$ to be a Rhie's polynomial, we get $\rho(\phi_t)=5n+m-6$.
\end{Theorem}
\begin{proof}
For sufficiently small $t$ and for each root $\al$ of $\ell_n$, by the continuity of the roots, there exists a root $\al(t)$ of $\phi_t=0$ in a neighborhood of  $\al$ with $\al(0)=\al$
and $\al(t)$ has the same orientation as $\al$. 
For $t\ne 0$, we know that $\be(\phi_t)=n-m$ for $t\ne 0$ and $\be(\ell_n)=n-1$.
Thus it is clear that we need at least $m-1$ negative roots.
Take a large $R>0$ so that $|\al_j|\le R/2$ for any $j=1,\dots, k$. For any small $\eps>0$,
there exists $\de(\eps)>0$ such that 
 $\phi_t$ has $k$ roots near each $\al_j(t),\,|t|\le \de(\eps)$ with 
the same sign as $\al_j$ in the original equation  $\ell_n=0$.
We may assume that $|\al_j(t)-\al_j|\le \eps$ for $j=1,\dots, k$ and there are no other roots of $\phi_t(z)=0$
in the disk $D_R=\{z\,|\, |z|\le R\}$.
 On the other hand, as
$\be(\phi_t)=n-m,\,t\ne 0$, we have the property $n-m\equiv k-(m-1)$, mod $2$.
Thus $\phi_t$ has at least $m-1$ new negative roots outside of the disk $D_R$.

We assert that $\phi_t$ obtains exactly $m-1$ new negative  roots near infinity.
To see this, we change the coordinate $u=1/z$ and 
dividing (\ref{eq-phi-t}) by ${\bar z}^m z^n$, we get
\begin{eqnarray}\label{eq-3}
\tilde\phi_t(u):=\frac{(t+\ga\bar u)^m-\ga^m{\bar u}^m}{\ga^{m-1}mt}\tilde q(u)-{\bar u}^m \tilde p(u)
\end{eqnarray}
 where $\tilde q,\tilde p$ are polynomials defined as 
$\tilde q(u)=u^nq(1/u),\, \tilde p(u)=u^np(1/u)$.
By the asumption $\deg\,q(z)=n$, we can write 
\[\begin{split}
\tilde q(u)&=\ga+\sum_{i=1}^n b_iu^i\\
\end{split}
\]
We will show that  for a sufficiently small $t>0$, there exist exactly $m-1$ roots  $u(t)$ which 
converges to 0 as $t\to 0$. 
Put $\tilde \phi_{t,1},\tilde\phi_{t,2}$ be the first and the second term of (\ref{eq-3}).
Putting  $u=vt$ for $t\ne 0$,
 we can write $\tilde \phi_{t,1}$ as 
\begin{eqnarray}
\tilde\phi_{t,1}(v)&:=&t^{m-1}\frac{(1+\bar v)^m-{\bar v}^m}{\ga^{m-1}m}\tilde q(vt)\\
&=&t^{m-1}h(v) \tilde q(vt)
\end{eqnarray}
where
$h(v)$ is a polynomial with a non-zero constant. $h(v)=0$ has $m-1$ simple roots, and we put them as $v=\beta_1,\dots,\beta_{m-1}$.
 Consider the disk at infinity and its subset $W$:
 \[\De:=\{u\,|\, |u|\le 1/R\},\, 
 W:=\{u\in \De\,|\,|v-\be_j|\ge \de,\,j=1,\dots, m-1\}.
 \]
On $\De$, we estimate
$1/M\le |\tilde q(vt)|\le M$ for some $M>0$.
Taking 
 a small number $\de>0$, we can make
\[\begin{split}
\left |\frac{(1+\bar v)^m-{\bar v}^m}{\ga^{m-1}m}\tilde q(vt)\right |\ge M'\de,\,\,v\in W
\end{split}
\]
with some constant $M'>0$. 
Or equivalently,
\[|\phi_{t,1}(u)|\ge |t|^{m-1}M'\de,\,u\in W.
\]
Taking $t$ small, we can make the second term of (\ref{eq-3}) as small as possible on $\De$
comparing with $|t|^{m-1}$. 
More precisely, there exists a positive number $M''$ such that 
\[
|\phi_{t,2}(u)|\le M'' |t|^{m}.
\]
Thus if  $|t|$ is sufficiently small, 
\[|\tilde \psi_t(v)|\ge \frac{M'}2 |t|^{m-1},\,\text{for}\, v\in W
\]
which implies $\tilde \psi_t(v)=0$ has one simple negative root in $D_j:=\{v\,|\, |v-\be_j|\le \de\}$
for $j=1,\dots, j$ and no root on $W$.
The negativity of these $m-1$ new roots is clear as $\beta(\phi_t)=n-m$ and $\beta(\ell_n)=n-1$.  This completes the proof.
\end{proof}
%
\subsection{Possible values of $\rho$}
Assume that $n\ge m$.
Combining Theorem 2, we can see that
\begin{Corollary} $\rho(f)$ for $f\in L(n+m;n,m)$ can takes 
the values
$\{n+m-2,\dots, 5n+m-6\}$.
\end{Corollary}
As for the lower  values $\{n-m,\dots, n+m-4\}$,
we know that these values can be taken by some polynomials in $M(n+m;n,m)$.
We do not know if these values can be taken in $L(n+m;n,m)$ or $L^{hs}(n+m;n,m)$ except 
$n-m$. For $n-m$, it can be taken by
\[
f(z,\bar z)={\bar z}^m z^n-1.
\]
\subsection{Example} Consider 
\[
f(z,\bar z)=\left(\frac{\bar z}{100}+1\right)^3(z^3-\frac 18)-z^3-\frac{3 z^2}{100}+\frac{12513}{100000}.
\]
This is a bifurcation of a Rhie type  polynomial $\ell_3(z,\bar z)$ with $\rho(\ell_3)=10$ where
\[
\ell_3(z,\bar z)=
\frac{3}{100}{\bar z}(z^3-\frac 18)-z^3-\frac{3 z^2}{100}+\frac{12513}{100000}=0.
\]
Then $\Re f=0$ is the green curve and $\Im f=0$ is the union of the real axis and the red curve.
The root $f=0$ is the intersection of two curves and we see 10 roots in the graph.
Actually there are two roots $w_1,w_2$ which are big :
\[
w_1\approx -150+86.6i,\quad w_2=\bar{w_1}.
\]
and thus $\rho(f)=12$.
\newpage
\begin{figure}[htb]
\setlength{\unitlength}{1bp}
\includegraphics[width=7cm,height=6cm]{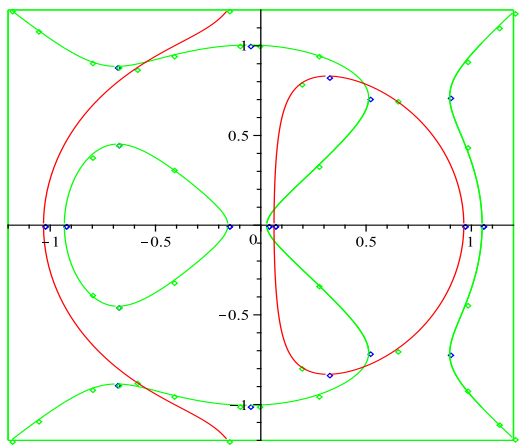}
\vspace{1cm}
\caption{Roots of $f(z,\bar z)=0$}
\end{figure}


\def\cprime{$'$} \def\cprime{$'$} \def\cprime{$'$} \def\cprime{$'$}
  \def\cprime{$'$} \def\cprime{$'$} \def\cprime{$'$} \def\cprime{$'$}

\end{document}